\documentclass[11pt,a4paper]{article}
\title{\bf Weak Type Estimates for Square Functions of Dunkl Heat Flows}
\author{Huaiqian Li\footnote{Email: {\color{blue}huaiqianlee@gmail.com}}
  \vspace{2mm}
\\
{\footnotesize Center for Applied Mathematics, Tianjin University, Tianjin 300072, P. R. China}
}

\date{}
\usepackage{amssymb,amsmath,amsfonts,amsthm,color,mathrsfs}
\usepackage{bbm}
\usepackage[marginal]{footmisc}

\usepackage{color} 

\def\R{\mathbb{R}}

\def\D{\mathbb{D}}

\def\L{\mathcal{L}}

\def\d{\textup{d}}
\def\D{\textup{D}}

\def\Lip{\textup{Lip}}

\def\RCD{\textup{RCD}}

\def\<{\langle}
\def\>{\rangle}
\def\Proof.{\noindent{\bf Proof. }}

\def\loc{\textup{loc}}

\def\newdot{{\kern.8pt\cdot\kern.8pt}}

\newtheorem{theorem}{Theorem}[section]
\newtheorem{lemma}[theorem]{Lemma}

\theoremstyle{definition}\newtheorem{remark}[theorem]{Remark}

\begin{document}
\allowdisplaybreaks
\maketitle
\makeatletter 
\renewcommand\theequation{\thesection.\arabic{equation}}
\@addtoreset{equation}{section}
\makeatother 

\begin{abstract}
The weak $(1,1)$ boundedness of the Littlewood--Paley--Stein square function for the Dunkl heat flow is proved via estimates on the Dunkl heat kernel of integral type and the Caldr\'{o}n--Zygmund decomposition, which is the continuity of the recently work \cite{LZ2020}
where the dimension-free $L^p$ boundedness of the same square function is studied.
\end{abstract}



\section{Introduction to main results}\hskip\parindent
In this section, we aim to recall some necessary basics on the Dunkl operator and then present the main results of this work. The Dunkl operator, initially introduced by C.F. Dunkl in \cite{Dunkl1988,Dunkl1989},  has been studied intensively. For a general overview on its development and  more details, refer to the survey papers \cite{Rosler2003,Anker2017} and the monographs \cite{DunklXu2014,DX2015}.

Consider the $d$-dimensional Euclidean space $\R^d$, endowed with the standard inner product $\langle\cdot,\cdot\rangle$ and the induced norm $|\cdot|$. For every $\alpha\in\R^d\setminus\{0\}$, define
$$r_\alpha x =x-2\frac{\langle \alpha,x\rangle}{|\alpha|^2}\alpha,\quad x\in\R^d,$$
where $r_\alpha$ is the reflection operator with respect to the hyperplane orthogonal to $\alpha$.

Let $\mathfrak{R}$ denote the root system, which is a finite subset of $\R^d\setminus\{0\}$ and satisfies that, for every $\alpha\in\mathfrak{R}$,  $r_\alpha\mathfrak{R}=\mathfrak{R}$ and $\alpha\R\cap \mathfrak{R}=\{\alpha,-\alpha\}$. Without loss of generality, we assume that $|\alpha|=\sqrt{2}$ for all $\alpha\in\mathfrak{R}$. Let $G$ be the reflection (or Weyl) group generated by $\{r_\alpha:\alpha\in\mathfrak{R}\}$. Note that $G$ is a finite subgroup of the orthogonal group $O(d)$, i.e., the group of $d\times d$ orthogonal matrices, and  $\{r_\alpha: \alpha\in\mathfrak{R}\}\subset G$ (see e.g. \cite[Theorem 6.2.7]{DunklXu2014} for a proof).  Let $\mathfrak{R}_+$ be any chosen positive subsystem such that $\mathfrak{R}$ is the disjoint union of $\mathfrak{R}_+$ and $-\mathfrak{R}_+$.

Let $\kappa_\cdot: \mathfrak{R}\rightarrow\R_+$ be the multiplicity function such that it is $G$-invariant, i.e., $\kappa_{g\alpha}=\kappa_\alpha$ for every $g\in G$ and every $\alpha\in\mathfrak{R}$.

Let $\xi\in\R^d$. Define the Dunkl operator $\D_\xi$ along $\xi$ associated to the root system $\mathfrak{R}$ and the multiplicity function $\kappa$  by
$$\D_\xi f(x)=\partial_\xi f(x)+\sum_{\alpha\in\mathfrak{R}_+}\kappa_\alpha \langle\alpha,\xi\rangle \frac{f(x)-f(r_\alpha x)}{\langle\alpha,x\rangle},\quad f\in C^1(\R^d),\,x\in\R^d,$$
where $\partial_\xi$ denotes the directional derivative along $\xi$. It is important to mention that, for every $\xi,\eta\in\R^d$, $\D_\eta\circ \D_\xi=\D_\xi\circ \D_\eta$. However, in general, due to the difference part, the Leibniz rule and the chain rule may not hold for $\D_\xi$.

Let $\{e_j: j=1,\cdots,d\}$ be the standard orthonormal basis of $\R^d$, and write $\D_j$ instead of $\D_{e_j}$ for short, $j=1,\cdots,d$. We denote $\nabla_\kappa=(\D_1,\cdots,\D_d)$ and $\Delta_\kappa=\sum_{j=1}^d \D_j^2$ the Dunkl gradient operator and the Dunkl Laplacian, respectively. By a straightforward calculation, for every $f\in C^2(\R^d)$,
\begin{eqnarray*}\label{Delta}
	\Delta_\kappa f(x)=\Delta f(x)+2\sum_{\alpha\in\mathfrak{R}_+}\kappa_\alpha\Big(\frac{\langle\alpha,\nabla f(x)\rangle}{\langle\alpha,x\rangle} - \frac{f(x)-f(r_\alpha x)}{\langle\alpha,x\rangle^2}\Big),\quad x\in\R^d.
\end{eqnarray*}
Obviously, when $\kappa=0$, then $\nabla_0=\nabla$ and $\Delta_0=\Delta$, which are the gradient operator and the Laplacian on $\R^d$, respectively.

Similar as the Laplacian case,  define the \emph{carr\'{e} du champ} (i.e., square (norm) of the (vector) field in English) $\Gamma$ (see e.g. \cite{BE1985}) by
$$\Gamma(f,g):=\frac{1}{2}\big[\Delta_\kappa(fg)-f\Delta_k g-g\Delta_\kappa f\big],\quad f,g\in C^2(\R^d).$$
Set $\Gamma(f)=\Gamma(f,f)$ for convenience. It is easy to show that, for every $f,g\in C^2(\R^d)$ and $x\in\R^d$,
\begin{equation}\label{Gamma-op}
\Gamma(f,g)(x)=\langle\nabla f(x),\nabla g(x)\rangle + \sum_{\alpha\in\mathfrak{R}_+}\kappa_\alpha\frac{\big(f(x)-f(r_\alpha x)\big)\big(g(x)-g(r_\alpha x)\big)}{\langle\alpha,x\rangle^2},
\end{equation}
and hence $\Gamma(f)\geq0$. Let
$$ \chi=\sum_{\alpha\in\mathfrak{R}_+}\kappa_\alpha.$$
From \cite[Remark 1.4(i)]{LZ2020}), we have the following pointwise inequality:
\begin{eqnarray}\label{grad-Gamma-comp}
|\nabla_\kappa f|^2\leq (1+2\chi)\Gamma(f),\quad f\in C^2(\R^d),
\end{eqnarray}
and in general, the converse is not true (see e.g. \cite[Theorem 3.5]{ML2019}).

The natural measure associated to the Dunkl operator is  $w_\kappa\L_d$, where for every $x\in\R^d$,
$$w(x):=\prod_{\alpha\in\mathfrak{R}_+}|\langle\alpha,x\rangle|^{2\kappa_\alpha},$$
and $\L_d$ stands for the Lebesgue measure on $\R^d$. Let $\mu_\kappa=w_\kappa\L_d$. For each $p\in[1,\infty]$, we denote the $L^p$ space by $L^p(\mu_\kappa):=L^p(\R^d,\mu_\kappa)$ and the corresponding norm by $\|\cdot\|_{L^p(\mu_\kappa)}$.

Let $H_\kappa(t):=e^{t\Delta_\kappa}$, $t\geq0$, be the Dunkl heat flow, which is self-adjoint in $L^2(\mu_\kappa)$. For $1\leq p<\infty$, $(H_\kappa(t))_{t\geq0}$ can be extended uniquely to a strongly continuous contraction semigroup in $L^p(\mu_\kappa)$, for which, with some abuse of notation, we keep the same notation. See \cite{Rosler1998,RosVoi1998,Rosler2003} for further properties.

We are concerned with square functions corresponding to the Dunkl heat flow. For $f\in C_c^\infty(\R^d)$, $x\in\R^d$, define the vertical Littlewood--Paley--Stein square functions by
\begin{eqnarray*}
\mathcal{V}_\Gamma(f)(x)&=&\bigg(\int_0^\infty \Gamma\big(H_\kappa(t)f\big)(x)\,\d t\bigg)^{1/2},\\
\mathcal{V}_{\nabla_\kappa}(f)(x)&=&\bigg(\int_0^\infty |\nabla_\kappa H_\kappa(t)f|^2(x)\,\d t\bigg)^{1/2},\\
\mathcal{V}_{\nabla}(f)(x)&=&\bigg(\int_0^\infty |\nabla H_\kappa(t)f|^2(x)\,\d t\bigg)^{1/2},
\end{eqnarray*}
and the horizontal Littlewood--Paley--Stein square function by
\begin{eqnarray*}
\mathcal{H}(f)(x)&=&\bigg(\int_0^\infty t\big|\partial_t H_\kappa(t)f\big|^2(x)\,\d t\bigg)^{1/2}.
\end{eqnarray*}
It is easy to show, as operators initially defined on $C_c^\infty(\R^d)$, $\mathcal{V}_{\Gamma},\mathcal{V}_{\nabla_\kappa},\mathcal{V}_{\nabla}$ and $\mathcal{H}$ are all sublinear.

In this work, we concentrate on the study of weak $(1,1)$ boundedness of the square functions defined above. From \eqref{grad-Gamma-comp} and the definition of $\Gamma$, we see that, for every $f\in C^\infty_c(\R^d)$, both $\mathcal{V}_{\nabla_\kappa}(f)$ and $\mathcal{V}_{\nabla}(f)$ are controlled by $\mathcal{V}_\Gamma(f)$ in the pointwise sense. Since the Dunkl heat flow $\{H_\kappa(t)\}_{t\geq0}$  is a symmetric diffusion semigroup in the sense of \cite[Page 65]{Stein70}, $\mathcal{H}$ is always bounded in $L^p(\mu_\kappa)$ for all $p\in (1,\infty)$ as a particular example of \cite[Corollary 1, page 120]{Stein70}. So, it is more interesting to us to study the weak $(1,1)$ boundedness of $\mathcal{V}_\Gamma$ and $\mathcal{H}$.

With these preparations in hand,  we can present our main results in the following theorems. The first one is on the vertical Littlewood--Paley--Stein square function.
\begin{theorem}\label{main-thm-v}
The operator $\mathcal{V}_{\Gamma}$ is weak $(1,1)$ bounded.
\end{theorem}

The second one is on the horizontal Littlewood--Paley--Stein square function.
\begin{theorem}\label{main-thm-h}
The operator $\mathcal{H}$ is weak $(1,1)$ bounded.
\end{theorem}

It is well known that square functions, which have been studied intensively, are one of the most fundamental concepts in harmonic analysis and play important roles in probability theory; see e.g. the survey paper \cite{Stein1982} and the book \cite{St1970}. Despite extensive studies of Littlewood--Paley--Stein square functions in various settings in the literature, we recall known results in the Dunkl setting. For $L^p$ boundedness,  see  \cite{Soltani2005} and \cite{LZL2017} in the one dimensional case, and see \cite{Soltani05,AmSi2012} and the recent \cite{LZ2020,DH2020} in high dimensional case. We mention that the results in the joint work \cite{LZ2020} are dimension-free, although restricted to the $\mathbb{Z}_2^d$ case when $p>2$. The weak $(1,1)$ boundedness seems not widely studied.  We should mention that, although the weak $(1,1)$ boundedness of square functions considered in \cite{DH2020} is not presented in the main result, the method via the vector valued Calder\'{o}n--Zygmund theory, which crucially depends on pointwise Dunkl heat kernel estimates and is different from the approach presented below, should imply the weak $(1,1)$ boundedness; see \cite[Proposition 3.1]{DH2020}.

Motivated by \cite{ADH2019} and \cite{CDD}, the idea to prove our Theorems \ref{main-thm-v} and \ref{main-thm-h} is the classic Caldr\'{o}n--Zygmund decomposition and  estimates on the Dunkl heat kernel of integral type. The same idea has been recently employed in \cite{AmHa2020} to prove the weak $(1,1)$ boundedness of the Riesz transform associated to the Dunkl--Schr\"{o}dinger operator $-\Delta_\kappa+V$ with $0\leq V\in L^2_{\loc}(\R^d)$ and to the Dunkl gradient operator $\nabla_\kappa$; see also \cite{AmHa2019} for more details on the Dunkl--Schr\"{o}dinger operator.

The present article is organized as follows. In Section 2, we recall necessary known facts and establish several lemmata that are important to prove our main results. In Section 3, we present the proofs of our main results.

We should point out that the constants $c,C,C',C'',\cdots$, used in what follows, may vary from one location to another.

\section{Preparations}\hskip\parindent
In this section, we recall necessary known facts and present some preliminary results which will be used to prove the main results. Let $B(x,r)$ denote the open ball in $\R^d$ with center $x\in \R^d$ and radius $r>0$, and for every $g\in G$ and every $A\subset\R^d$, let $gA=\{gx\in\R^d: x\in A\}$.

Let $d_\kappa=d+2\chi$.  It is known that $\mu_\kappa$ is $G$-invariant, i.e., for every $g\in G$ and every ball $B\subset\R^d$, $\mu_\kappa(gB)=\mu_\kappa(B)$, and the volume comparison property (see e.g. \cite[(3.2)]{ADH2019}) holds: there is a constant $\theta\geq1$ such that, for every $x\in\R^d$ and every $0<r\leq R<\infty$,
\begin{eqnarray}\label{doubling}
\frac{1}{\theta}\Big(\frac{R}{r}\Big)^{d}\leq\frac{\mu\big(B(x,R)\big)}{\mu\big(B(x,r)\big)}\leq \theta\Big(\frac{R}{r}\Big)^{d_\kappa}.
\end{eqnarray}
However, we do not use the left inequality in the proofs.

Let $\xi\in \R^d$. With respect to $\mu_\kappa$, the following integration-by-parts formula holds: for every $u\in C^1(\R^d)$ and every $v\in C_c^1(\R^d)$,
\begin{eqnarray}\label{anti-sym}
\int_{\R^d}v\D_\xi u\,\d\mu_\kappa=-\int_{\R^d}u \D_\xi v\,\d\mu_\kappa.
\end{eqnarray}
See \cite[Lemma 2.9]{Dunkl1992} and \cite[Proposition 2.1]{Rosler2003}. It is easy to see that \eqref{anti-sym} holds true when $C^1(\R^d)$ is replaced by $\Lip_{\loc}(\R^d)$, the space of locally Lipschitz continuous function on $\R^d$ with respect to the Euclidean distance $|\cdot - \cdot|$. Although we may not expect that the Dunkl operator satisfies the Leibniz rule in general, the following particular case is useful (see e.g. \cite[(2.1)]{Rosler2003} and see \cite[Proposition 6.4.12]{DunklXu2014} for the general situation): for every $u,v\in C^1(\R^d)$ with  at lest one of them being $G$-invariant,
\begin{eqnarray}\label{lebniz}
\D_\xi(uv)=v\D_\xi u+u\D_\xi v.
\end{eqnarray}

For every $x\in\R^d$, let $G(x)=\{gx: g\in G\}$, which denotes the $G$-orbit of $x$. Let
$$\rho(x,y)=\min_{g\in G}|x-gy|,\quad x,y\in\R^d,$$
which is the distance between $G$-orbits $G(x)$ and $G(y)$. Note also that $\rho$ is $G$-invariant for each variable by definition. However, $\rho$ may not be a pseudo-distance (or more standardly called quasi-metric) on $\R^d\times\R^d$  in the sense of \cite[Page 66]{CW1971}, and hence the triple $(\R^d,\rho,\mu_\kappa)$ should not be regarded as a space of homogeneous type studied extensively in harmonic analysis. Moreover, the following small observation is useful. For any point $x_0\in\R^d$, we let $\rho_{x_0}(\cdot)=\rho(x_0,\cdot)$.
\begin{lemma}\label{grad-dist}
For an arbitrarily fixed point $x_0\in\R^d$,
\begin{equation*}\label{grad-dist-1}
|\nabla\rho_{x_0}(x)|\leq1,\quad\mbox{for }\mu_\kappa\mbox{-a.e. }x\in\R^d.
\end{equation*}
\end{lemma}
\begin{proof}
By the definition of $\rho$, we have
$$|\rho_{x_0}(y)-\rho_{x_0}(z)|\leq |y-z|, \quad y,z\in\R^d,$$
which implies that $\rho_{x_0}(\cdot)$ is Lipschitz continuous with respect to $|\cdot-\cdot|$ with Lipschitz constant $1$. Then,  by the well-known  Rademacher theorem,  $\rho_{x_0}(\cdot)$ is differentiable almost everywhere with respect to $\L_d$; furthermore,
 $$|\nabla\rho_{x_0}(x)|\leq1,\quad\mbox{for }\L_d\mbox{-a.e. }x\in\R^d.$$
 Since $\mu_\kappa$ is clearly absolutely continuous with respect to $\L_d$, we complete the proof.
\end{proof}

Let $h_t(x,y)$ be the Dunkl heat kernel of $H_\kappa(t)$, which is a $C^\infty$ function of all variables $x,y\in\R^d$ and $t>0$, and satisfies that
\begin{eqnarray*}
&&\partial_t h_t(x,y)=\Delta_\kappa h_t(\cdot,y)(x),\quad x,y\in\R^d,\,t>0,\\
&&h_t(x,y)=h_t(y,x)>0,\quad x,y\in\R^d,\,t>0,
\end{eqnarray*} and, moreover, there exist positive constants $c,C$ such that
\begin{eqnarray}\label{kernel-bd}
h_t(x,y)\leq \frac{C}{V(x,y,t)}\exp\Big(-c\frac{\rho(x,y)^2}{t}\Big),\quad x,y\in\R^d,\,t>0.
\end{eqnarray}
Here and in what follows, we use the notation
$$V(x,y,r)=\max\big\{\mu_\kappa\big(B(x,r)\big),\mu_\kappa\big(B(y,r)\big)\big\}.$$
See e.g. \cite{Rosler2003} for more details on the Dunkl heat kernel. Recently, the following estimate on time derivative of the Dunkl heat kernel is established in \cite[Theorem 4.1(a)]{ADH2019}: for every nonnegative integer $m$, there exist positive constants $c,C$ such that
\begin{eqnarray}\label{time-kernel-bd}
| \partial_t^m h_t(x,y)|\leq \frac{c}{t^{m}V(x,y,t)}\exp\Big(-C\frac{\rho(x,y)^2}{t}\Big),\quad x,y\in\R^d,\,t>0,
\end{eqnarray}
whose proof employs the integral representation of the Dunkl translation operator first obtained in the paper \cite{Rosler2003a} (see also \cite[Lemma 3.4]{DaiWang2010}). However, we should give a remark here.
\begin{remark}
Although $\rho$ may not be a true metric, by the analyticity of $t\mapsto h_t(x,y)$, estimate \eqref{kernel-bd} and  the right inequality of \eqref{doubling},   it is possible to obtain \eqref{time-kernel-bd} in another way by applying the general result \cite[Theorem 4]{Davies1997} (whose proof does not depend on the metric structure). See also the recent paper \cite{DziPre2018} for the homogeneous space setting.
\end{remark}

Let $|G|$ denote the order of the reflection group $G$. For $x\in\R^d$ and $r\geq0$, define
$$B^\rho(x,r)=\{y\in\R^d: \rho(x,y)<r\},$$
where $B^\rho(x,0):=\{y\in\R^d: \rho(x,y)=0\}$ and it is at most a finite subset of $\R^d$. From the volume comparison property \eqref{doubling} and the Dunkl heat kernel estimate \eqref{kernel-bd}, we can immediately obtain the following lemma. The proof is standard and short, and we present it here for the sake of completeness (see e.g. the proof of \cite[Lemma 2.1]{CoulhonDuong1999}).
\begin{lemma}\label{lemma-1}
For every $\delta>0$, there exists a positive constant $C$ such that
$$\int_{\R^d\setminus B^\rho(y,\sqrt{t})}\exp\Big(-2\delta\frac{\rho(x,y)^2}{s}\Big)\,\d\mu_\kappa(x)\leq C \mu_\kappa\big(B(y,\sqrt{s})\big)e^{-\delta t/s},$$
for every $s>0$, $t\geq0$ and $y\in\R^d$.
\end{lemma}
\begin{proof} Let ${\rm I}=\int_{\R^d}e^{-\delta\rho(x,y)^2/s}\,\d\mu_\kappa(x)$. Then
\begin{eqnarray*}
{\rm I}&=&\sum_{n=0}^\infty\int_{B^\rho(y,(n+1)\sqrt{s})\setminus B^\rho(y,n\sqrt{s})}e^{-\delta\rho(x,y)^2/s}\,\d\mu_\kappa(x)\\
&\leq&\sum_{n=0}^\infty e^{-\delta n^2}\mu_{\kappa}\big(B^\rho(y,(n+1)\sqrt{s})\big).
\end{eqnarray*}
Since for $x\in\R^d$ and $r>0$,
\begin{eqnarray*}
B^\rho(x,r)&=&\cup_{g\in G}\{y\in\R^d:|x-gy|<r\}
=\cup_{g\in G}g B(x,r),
 \end{eqnarray*}
we have, by the $G$-invariance of $\mu_\kappa$ and the right inequality of \eqref{doubling},
\begin{eqnarray*}
{\rm I}&\leq&\sum_{n=0}^\infty e^{-\delta n^2}\mu_{\kappa}\big(\cup_{g\in G} gB(y,(n+1)\sqrt{s})\big)\leq \sum_{n=0}^\infty e^{-\delta n^2}|G|\mu_{\kappa}\big(B(y,(n+1)\sqrt{s})\big)\\
&\leq&|G|\sum_{n=0}^\infty e^{-\delta n^2}(n+1)^{d_\kappa}\mu_\kappa\big(B(y,\sqrt{s})\big)\leq C\mu_\kappa\big(B(y,\sqrt{s})\big).
\end{eqnarray*}
Thus
$$\int_{\R^d\setminus B^\rho(y,\sqrt{t})}\exp\Big(-2\delta\frac{\rho(x,y)^2}{s}\Big)\,\d\mu_\kappa(x)\leq e^{-\delta t/s}\, {\rm I}
\leq C\mu_\kappa\big(B(y,\sqrt{s})\big)e^{-\delta t/s},$$
which completes the proof of Lemma \ref{lemma-1}.
\end{proof}

The next result is on the integral type of gradient estimate of the Dunkl heat kernel, which is motivated by \cite[Lemma 3.3]{CDD}. However, due to the lack of the Leibniz rule and the chain rule for the Dunkl Laplacian, the method used in the aforementioned reference is no longer directly applicable.
\begin{lemma}\label{lemma-2}
For every nonnegative integer $m$ and for small enough $\epsilon>0$, there exists a positive constant $c_\epsilon$ such that
\begin{eqnarray}\label{int-Gamma-kernel-bd-1}
\int_{\R^d}\Gamma\big(\Delta_{\kappa}^m h_s(\cdot,y)\big)(x)\,\d\mu_\kappa(x)\leq\frac{c_\epsilon}{s^{2m+1}\mu_\kappa\big(B(y,\sqrt{s})\big)},
\end{eqnarray}
and
\begin{eqnarray}\label{int-Gamma-kernel-bd-2}
&&\int_{\R^d\setminus B^\rho(y,\sqrt{t})}\Gamma\big(\Delta_{\kappa}^m h_s(\cdot,y)\big)(x)\exp\Big(\epsilon\frac{\rho(x,y)^2}{s}\Big)\,\d\mu_\kappa(x)
\leq\frac{c_\epsilon \,e^{-\epsilon t/s}}{s^{2m+1}\mu_\kappa\big(B(y,\sqrt{s})\big)},
\end{eqnarray}
for all $y\in\R^d,\,s>0,\,t\geq0$.
\end{lemma}
\begin{proof}
Let $x,y\in\R^d$,  $\epsilon,s,R>0$ and let $m$ be a nonnegative integer. For every $\alpha\in\mathfrak{R}$, let $\alpha=(\alpha_1,\cdots,\alpha_d)$. For convenience, we let $f(x)=\partial_s^m h_s(x,y)$, $\eta(x)=e^{2\epsilon\rho(x,y)^2/s}$. Then $f\in C^\infty(\R^d)$ and $\eta\in \Lip_{\loc}(\R^d)$. Take
$$\phi_R(x)=\min\Big\{1,\Big(3-\frac{|x|}{R}\Big)^+\Big\},\quad x\in\R^d,$$
where for any $a\in\R$, $a^+:=\max\{a,0\}$. Then, $0\leq\phi_R\leq1$ on $\R^d$, $\phi_R=1$ on $B(0,2R)$, $\phi_R=0$ outside $B(0,3R)$; moreover, $\phi_R$ is Lipschitz continuous with respect to $|\cdot - \cdot|$, $G$-invariant, increasing as $R$ grows up and $|\nabla \phi_R|\leq 1/R$. Note that $\eta$ is $G$-invariant. Hence, $\eta\phi_R$ is $G$-invariant and $\eta\phi_R\in\Lip_{\loc}(\R^d)$.  Set
$$J=\int_{\R^d}\Gamma(f)\eta\,\d\mu_\kappa,\quad\quad J_R=\int_{\R^d}\Gamma(f)\eta\phi_R^2\,\d\mu_\kappa,$$
 and
 $$J_{R,1}=\frac{1}{2}\int_{\R^d}\Delta_\kappa(f^2)\eta\phi_R^2\,\d\mu_\kappa,\quad\quad J_{R,2}=-\int_{\R^d}f\langle\nabla_\kappa f,\nabla_\kappa(\eta\phi_R^2)\rangle\,\d\mu_\kappa.$$

 By \eqref{anti-sym} and \eqref{lebniz}, we have
\begin{eqnarray*}
&&J_{R,1}\\
&=&-\frac{1}{2}\sum_{j=1}^d\int_{\R^d}\D_j(f^2)\partial_j(\eta\phi_R^2)\,\d\mu_\kappa\\
&&\times\big[\phi_R^2(x)\partial_j\eta(x)+\eta(x)\partial_j\phi_R^2(x)\big]\,\d\mu_\kappa(x)\\
&=&-\int_{\R^d}\Big[f(x)\langle \nabla f(x),\nabla\eta(x)\rangle+\frac{1}{2}\sum_{\alpha\in\mathfrak{R}_+}\kappa_\alpha\langle\alpha,\nabla\eta(x)\rangle\frac{f^2(x)-f^2(r_\alpha x)}{\langle\alpha,x\rangle}\Big]\phi_R^2(x)\,\d\mu_\kappa(x)\\
&&-\int_{\R^d}\Big[f(x)\langle \nabla f(x),\nabla\phi_R^2(x)\rangle+\frac{1}{2}\sum_{\alpha\in\mathfrak{R}_+}\kappa_\alpha\langle\alpha,\nabla\phi_R^2(x)\rangle\frac{f^2(x)-f^2(r_\alpha x)}{\langle\alpha,x\rangle}\Big]\eta(x)\,\d\mu_\kappa(x),
\end{eqnarray*}
and
\begin{eqnarray*}
&&J_{R,2}\\
&=&-\int_{\R^d}\Big(f(x)\langle \nabla f(x),\nabla\eta(x)\rangle+\sum_{\alpha\in\mathfrak{R}_+}\kappa_\alpha\langle\alpha,\nabla\eta(x)\rangle\frac{f(x)[f(x)-f(r_\alpha x)]}{\langle\alpha,x\rangle}\Big)\phi_R^2(x)\,\d\mu_\kappa(x)\\
&&-\int_{\R^d}\Big(f(x)\langle \nabla f(x),\nabla\phi_R^2(x)\rangle+\sum_{\alpha\in\mathfrak{R}_+}\kappa_\alpha\langle\alpha,\nabla\phi_R^2(x)\rangle\frac{f(x)[f(x)-f(r_\alpha x)]}{\langle\alpha,x\rangle}\Big)\eta(x)\,\d\mu_\kappa(x).
\end{eqnarray*}
Then
\begin{eqnarray}\label{J1-J2}
J_{R,1}-J_{R,2}&=&\frac{1}{2}\int_{\R^d}\sum_{\alpha\in\mathfrak{R}_+}\kappa_\alpha\langle\alpha,\nabla\eta(x)\rangle\frac{[f(x)-f(r_\alpha x)]^2}{\langle \alpha,x\rangle}\phi_R^2(x)\,\d\mu_\kappa(x)\cr
&&+\frac{1}{2}\int_{\R^d}\sum_{\alpha\in\mathfrak{R}_+}\kappa_\alpha\langle\alpha,\nabla\phi_R^2(x)\rangle\frac{[f(x)-f(r_\alpha x)]^2}{\langle \alpha,x\rangle}\eta(x)\,\d\mu_\kappa(x)\cr
&=:&A_R+B_R.
\end{eqnarray}

Combing \eqref{time-kernel-bd} with the same method used to prove \cite[(4.12)]{ADH2019}, we obtain the following estimate
$$\frac{\big[f(x)-f(r_\alpha x)\big]^2}{|\langle \alpha,x\rangle|}\leq\frac{c}{s^{2m+1/2}\mu_\kappa\big(B(y,\sqrt{s})\big)^2}\exp\Big(-{C}\frac{\rho(x,y)^2}{s}\Big).$$
Since $0\leq\phi_R\leq1$ and $|\nabla \phi_R|\leq 1/R$, by Lemma \ref{grad-dist} and Lemma \ref{lemma-1}, we derive that, for small enough $\epsilon$,
\begin{eqnarray}\label{lemma-2-A}
|A_R|&\leq&\frac{c}{s^{2m+1/2}\mu_\kappa\big(B(y,\sqrt{s})\big)^2}\int_{\R^d}\frac{\rho(x,y)}{s}\exp\Big(2\epsilon\frac{\rho(x,y)^2}{s}\Big)
\exp\Big(-{C}\frac{\rho(x,y)^2}{s}\Big)\,\d\mu_\kappa(x)\cr
&\leq&\frac{c}{s^{2m+1}\mu_\kappa\big(B(y,\sqrt{s})\big)^2}\int_{\R^d}\exp\Big(-{(C'-2\epsilon)}\frac{\rho(x,y)^2}{s}\Big)\,\d\mu_\kappa(x)\cr
&\leq&\frac{c}{s^{2m+1}\mu_\kappa\big(B(y,\sqrt{s})\big)},
\end{eqnarray}
and
\begin{eqnarray}\label{lemma-2-B}
|B_R|&\leq&\frac{c}{R s^{2m+1/2}\mu_\kappa\big(B(y,\sqrt{s})\big)^2}\int_{\R^d}\exp\Big(-{(C-2\epsilon)}\frac{\rho(x,y)^2}{s}\Big)\,\d\mu_\kappa(x)\cr
&\leq&\frac{c}{R s^{2m+1/2}\mu_\kappa\big(B(y,\sqrt{s})\big)},
\end{eqnarray}
which tends to $0$ as $R\rightarrow\infty$.  Thus, from \eqref{J1-J2} and \eqref{lemma-2-A}, we have
\begin{eqnarray}\label{lemma-2-1}
|J_{R,1}|&\leq&|J_{R,2}|+|A_R|+|B_R|\cr
&\leq&|J_{R,2}|+\frac{c}{s^{2m+1}\mu_\kappa\big(B(y,\sqrt{s})\big)}+|B_R|.
\end{eqnarray}

To estimate $J_{R,2}$, we deduce that
\begin{eqnarray*}
|J_{R,2}|&=&\Big|\int_{\R^d}f\langle \nabla_\kappa f,\nabla \eta\rangle \phi_R^2\,\d\mu_\kappa + \int_{\R^d}f\langle \nabla_\kappa f,\nabla\phi_R^2\rangle \eta\,\d\mu_\kappa \Big|\\
&\leq&\int_{\R^d}|f||\nabla_\kappa f||\nabla\eta|\phi_R^2\,\d\mu_\kappa + \frac{2}{R}\int_{\R^d}|f||\nabla_\kappa f|\eta\phi_R\,\d\mu_\kappa\\
&=:& J_{R,2,1}+J_{R,2,2}.
\end{eqnarray*}
For the estimation of $ J_{R,2,1}$, we have
\begin{eqnarray*}
 J_{R,2,1}&\leq& \int_{\R^d}|f(x)||\nabla_\kappa f(x)|\frac{4\epsilon\rho(x,y)}{s}\exp\Big(2\epsilon\frac{\rho(x,y)^2}{s}\Big)\phi_R^2(x)\,\d\mu_\kappa(x)\\
&\leq&\frac{c}{\sqrt{s}}\int_{\R^d}|f(x)||\nabla_\kappa f(x)|\exp\Big(\epsilon'\frac{\rho(x,y)^2}{s}\Big)\phi_R(x)\,\d\mu_\kappa(x)\\
&\leq&\frac{c}{\sqrt{s}}\bigg(\int_{\R^d}|f(x)|^2e^{\epsilon''\rho(x,y)^2/s}\,\d\mu_\kappa(x)\bigg)^{1/2}\\
&&\times\bigg(\int_{\R^d}|\nabla_\kappa f(x)|^2e^{2\epsilon\rho(x,y)^2/s}\phi_R^2(x)\,\d\mu_\kappa(x)\bigg)^{1/2},
\end{eqnarray*}
where we used Lemma \ref{grad-dist} and $0\leq\phi_R\leq1$ again in the second inequality, and the Cauchy--Schwarz inequality in the last inequality. By Lemma \ref{lemma-1} and \eqref{time-kernel-bd}, it is easy to see that, for small enough $\epsilon$,
$$\int_{\R^d}|f(x)|^2e^{\epsilon''\rho(x,y)^2/s}\,\d\mu_\kappa(x)\leq\frac{c}{s^{2m}\mu_\kappa\big(B(y,\sqrt{s})\big)}.$$
By the pointwise inequality \eqref{grad-Gamma-comp}, we have
$$\int_{\R^d}|\nabla_\kappa f(x)|^2e^{2\epsilon\rho(x,y)^2/s}\phi_R^2(x)\,\d\mu_\kappa(x)\leq(1+2\chi)J_R.$$
Hence
\begin{eqnarray*}\label{lemma-2-2}
J_{R,2,1}\leq \frac{c \sqrt{J_R}}{\sqrt{s^{2m+1}\mu_\kappa\big(B(y,\sqrt{s})\big)}}.
\end{eqnarray*}
For the estimation of $J_{R,2,2}$, we have
\begin{eqnarray*}\label{}
J_{R,2,2}&\leq& \frac{2}{R}\Big(\int_{\R^d}|f|^2\eta\,\d\mu_\kappa\Big)^{1/2}\Big(\int_{\R^d}|\nabla_\kappa f|^2\eta\phi_R^2\,\d\mu_\kappa\Big)^{1/2}\cr
&\leq& \frac{2}{R}\Big(\int_{\R^d}|f|^2\eta\,\d\mu_\kappa\Big)^{1/2}\Big((1+2\chi)\int_{\R^d}\Gamma(f)\eta\phi_R^2\,\d\mu_\kappa\Big)^{1/2}\cr
&\leq& \frac{2(1+2\chi)}{R^2}\int_{\R^d}|f|^2\eta\,\d\mu_\kappa +\frac{1}{2}J_R\\
&\leq&\frac{c}{R^2 s^{2m}\mu_\kappa\big(B(y,\sqrt{s})\big)}+\frac{1}{2}J_R,
\end{eqnarray*}
where we used \eqref{grad-Gamma-comp}, Lemma \ref{lemma-1}, \eqref{time-kernel-bd} and Young's inequality. Combing the estimates of $J_{R,2,1}$ and $J_{R,2,2}$, we obtain
\begin{eqnarray}\label{lemma-2-2}
|J_{R,2}|\leq \frac{c \sqrt{J_R}}{\sqrt{s^{2m+1}\mu_\kappa\big(B(y,\sqrt{s})\big)}}+ \frac{c}{R^2 s^{2m}\mu_\kappa\big(B(y,\sqrt{s})\big)}+\frac{1}{2}J_R.
\end{eqnarray}

By applying \eqref{time-kernel-bd} and Lemma \ref{lemma-1} again, we get that, for small enough $\epsilon$,
\begin{eqnarray}\label{lemma-2-4}
\bigg|\int_{\R^d}(f\Delta_\kappa f )\eta\phi_R^2\,\d\mu_\kappa\bigg|
&\leq&\int_{\R^d}|\partial_s^m h_s(x,y)||\partial_s^{m+1} h_s(x,y)|e^{2\epsilon\rho(x,y)^2/s}\,\d\mu_\kappa(x)\cr
&\leq&\frac{c}{s^{2m+1}\mu_\kappa\big(B(y,\sqrt{s})\big)}.
\end{eqnarray}

Thus, combing \eqref{lemma-2-1}, \eqref{lemma-2-2} and \eqref{lemma-2-4}, we have
\begin{eqnarray*}\label{lemma-2-5}
J_R&=&\frac{1}{2}\int_{\R^d}\Delta_\kappa(f^2)\eta\phi_R^2\,\d\mu_\kappa-\int_{\R^d}f(\Delta_\kappa f)\eta\phi_R^2\,\d\mu_\kappa\\
&\leq&|J_{R,1}|+\frac{c}{s^{2m+1}\mu_\kappa\big(B(y,\sqrt{s})\big)}\\
&\leq&\frac{C}{s^{2m+1}\mu_\kappa\big(B(y,\sqrt{s})\big)}+\frac{1}{2}J_R+\frac{c \sqrt{J_R}}{\sqrt{s^{2m+1}\mu_\kappa\big(B(y,\sqrt{s})\big)}}\\
&&+\frac{c}{R^2 s^{2m}\mu_\kappa\big(B(y,\sqrt{s})\big)}+|B_R|.
\end{eqnarray*}
By \eqref{lemma-2-B} and the monotone convergence theorem, letting $R\rightarrow\infty$,  we obtain
$$J\leq \frac{C}{s^{2m+1}\mu_\kappa\big(B(y,\sqrt{s})\big)}+\frac{c \sqrt{J}}{\sqrt{s^{2m+1}\mu_\kappa\big(B(y,\sqrt{s})\big)}},$$
which immediately implies that
$$J\leq\frac{C}{s^{2m+1}\mu_\kappa\big(B(y,\sqrt{s})\big)}.$$
We complete the proof of \eqref{int-Gamma-kernel-bd-1}.

Finally, for every $t\geq0$,
\begin{eqnarray*}
&&\int_{\R^d\setminus B^\rho(y,\sqrt{t})}\Gamma\big(\Delta_{\kappa}^mh_s(\cdot,y)\big)(x)
\exp\Big(\epsilon\frac{\rho(x,y)^2}{s}\Big)\,\d\mu_\kappa(x)\\
&=&\int_{\R^d\setminus B^\rho(y,\sqrt{t})}\Gamma\big(\Delta_{\kappa}^mh_s(\cdot,y)\big)(x)\exp\Big(2\epsilon\frac{\rho(x,y)^2}{s}\Big)
\exp\Big(-\epsilon\frac{\rho(x,y)^2}{s}\Big)\,\d\mu_\kappa(x)\\
&\leq&e^{-\epsilon t/s}J\leq\frac{Ce^{-\epsilon t/s}}{s^{2m+1}\mu_\kappa\big(B(y,\sqrt{s})\big)},
\end{eqnarray*}
which completes the proof of \eqref{int-Gamma-kernel-bd-2}.
\end{proof}

Now we should give a remark on the proof of Lemma \ref{lemma-2}.
\begin{remark} Recently, the following pointwise estimate on space-time derivative of the Dunkl heat kernel is established in \cite[Theorem 4.1(c)]{ADH2019}: for every $j=1,\cdots,d$ and every nonnegative integer $m$, there exist positive constants $c,C$ such that
\begin{eqnarray}\label{mixed-kernel-bd}
|\D_j \partial_t^m h_t(\cdot,y)|(x)\leq \frac{c}{t^{m+1/2}V(x,y,\sqrt{t})}\exp\Big(-C\frac{\rho(x,y)^2}{t}\Big),\quad x,y\in\R^d,\,t>0.
\end{eqnarray}
Applying the same method used to obtain \eqref{mixed-kernel-bd} (see the proof of \cite[Theorem 4.1(c)]{ADH2019}), we can obtain the following pointwise gradient bound on the Dunkl heat kernel, which seems stronger than \eqref{mixed-kernel-bd} due to the pointwise bound \eqref{grad-Gamma-comp} and its converse is not true in general. For every nonnegative integer $m$,  there exist positive constants $c_1, c_2$ such that
\begin{eqnarray*}\label{Gamma-kernel-bd}
\sqrt{\Gamma\big(\Delta_\kappa^m h_t(\cdot,y)\big)(x)}\leq \frac{c_1}{t^{m+1/2}V(x,y,\sqrt{t})}\exp\Big(-c_2\frac{\rho(x,y)^2}{t}\Big),\quad x,y\in\R^d,\,t>0.
\end{eqnarray*}
Then, applying Lemma \ref{lemma-1}, we can also obtain Lemma \ref{lemma-2}. This approach seems more straightforward in the present situation.  However, in other settings, for instance on curved spaces, pointwise gradient kernel bounds are not easy to get, which demand geometric conditions usually, for instance, Ricci curvature bounded from below on Riemannian manifolds (see e.g. \cite{JLZ2016} for the more general case on $\RCD$ spaces). Our approach to prove Lemma \ref{lemma-2} above has the advantage that we may establish the gradient kernel bound of integral type, say \eqref{int-Gamma-kernel-bd-1}, even without the pointwise gradient kernel bound.
\end{remark}

In order to obtain the weak $(1,1)$ boundedness of the horizontal square function $\mathcal{H}$, we need the following lemma, which can be easily verified by applying Lemma \ref{lemma-1} with the estimate \eqref{time-kernel-bd} in hand.
\begin{lemma}\label{integral-time-der-bd}
For every nonnegative integer $m$ and for small enough $\epsilon>0$, there exists a positive constant $C_\epsilon$ such that
\begin{eqnarray*}\label{integral-time-der-bd-1}
&&\int_{\R^d\setminus B^\rho(y,\sqrt{t})}|\Delta_{\kappa}^m h_s(\cdot,y)|^2(x)\exp\Big(\epsilon\frac{\rho(x,y)^2}{s}\Big)\,\d\mu_\kappa(x)
\leq\frac{C_\epsilon\,e^{-\epsilon t/s}}{s^{2m}\mu_\kappa\big(B(y,\sqrt{s})\big)},
\end{eqnarray*}
for all $y\in\R^d,\,s>0,\,t\geq0$.
\end{lemma}

\section{Proofs of the main results}\hskip\parindent
Now we are in a position to prove the main results.
\begin{proof}[Proof of Theorem \ref{main-thm-v}] Let $f\in L^1(\mu_\kappa)$ and $\lambda>0$. By the classical Cader\'{o}n--Zygmund decomposition, we have
$$f=g+\sum_i b_i=:g+b,$$
and the following assertions hold: there exists a positive constant $c$ such that
\begin{itemize}
\item[(a)] $|g(x)|\leq c\lambda$ for $\mu_\kappa$-a.e. $x\in\R^d$,
\item[(b)] there exists a sequence of balls $\{B_i\}_i$ in $\R^d$ with $B_i=B(x_i,r_i)$ such that $r_i\in(0,1]$, $x_i\in\R^d$, $b_i$ is supported in $B_i$ and $\|b_i\|_{L^1(\mu_\kappa)}\leq c\lambda \mu_\kappa(B_i)$ for each $i$,
\item[(c)]  $\sum_{i}\mu_\kappa(B_i)\leq c\lambda^{-1}\|f\|_{L^1(\mu_\kappa)}$,
\item[(d)] every point of $\R^d$ is contained in at most finitely many balls $B_i$.
\end{itemize}

We shall prove that
\begin{eqnarray}\label{proof-main-0}
\mu_\kappa(\{x\in\R^d: \mathcal{V}_{\Gamma}(f)(x)\geq \lambda\})\leq \frac{c}{\lambda}\|f\|_{L^1(\mu_\kappa)}.
\end{eqnarray}

 By $(b),(c)$, we immediately get $\|b\|_{L^1(\mu_\kappa)}\leq\sum_i\|b_i\|_{L^1(\mu_\kappa)}\leq c\|f\|_{L^1(\mu_\kappa)}$, and hence, $\|g\|_{L^1(\mu_\kappa)}\leq c\|f\|_{L^1(\mu_\kappa)}$.

We divide the proof into four parts.

\textbf{(1)} By the sublinearity of $f\mapsto\mathcal{V}_{\Gamma}(f)$ and the decomposition of $f$, we have
\begin{eqnarray}\label{proof-main-1}
&&\mu_\kappa(\{x\in\R^d: \mathcal{V}_{\Gamma}(f)(x)\geq \lambda\})\cr
&\leq& \mu_\kappa(\{x\in\R^d: \mathcal{V}_{\Gamma}(g)(x)\geq \lambda/2\})+\mu_\kappa(\{x\in\R^d: \mathcal{V}_{\Gamma}(b)(x)\geq \lambda/2\}).
\end{eqnarray}
Since $\mathcal{V}_\Gamma$ is bounded in $L^2(\mu_\kappa)$ (see \cite[Theorem 2.4]{LZ2020}), by (a) and Chebyshev's inequality, we have
\begin{eqnarray}\label{proof-main-2}
\mu_\kappa(\{x\in\R^d: \mathcal{V}_{\Gamma}(g)(x)\geq \lambda/2\})&\leq& \frac{c}{\lambda^2}\|\mathcal{V}_{\Gamma}(g)\|_{L^2(\mu_\kappa)}^2\leq
\frac{c}{\lambda^2}\|g\|_{L^2(\mu_\kappa)}^2\cr
&\leq&\frac{c}{\lambda}\|f\|_{L^1(\mu_\kappa)}.
\end{eqnarray}

\textbf{(2)} Let $t_i=r_i^2$ and $I$ be the identity map. Since
\begin{eqnarray*}
\mathcal{V}_{\Gamma}(b_i)&=&\mathcal{V}_{\Gamma}\big(H_\kappa(t_i)b_i+ [I-H_\kappa(t_i)]b_i\big)\\
&\leq&\mathcal{V}_{\Gamma}\big(H_\kappa(t_i)b_i)+\mathcal{V}_{\Gamma}\big([I-H_\kappa(t_i)]b_i),
\end{eqnarray*}
we have
\begin{eqnarray*}
\mathcal{V}_{\Gamma}(b)=\mathcal{V}_{\Gamma}\Big(\sum_i b_i\Big)\leq\mathcal{V}_{\Gamma}\Big(\sum_i H_\kappa(t_i)b_i\Big)+\sum_i\mathcal{V}_{\Gamma}\big([I-H_\kappa(t_i)]b_i).
\end{eqnarray*}
Then
\begin{eqnarray}\label{proof-main-3}
&&\mu_\kappa(\{x\in\R^d: \mathcal{V}_{\Gamma}(b)(x)\geq \lambda/2\})\cr
&\leq&\mu_\kappa\Big(\Big\{x\in\R^d: \mathcal{V}_{\Gamma}\Big(\sum_i H_\kappa(t_i)b_i\Big)(x)\geq \lambda/4\Big\}\Big)\cr
&&+\mu_\kappa\Big(\Big\{x\in\R^d: \sum_i\mathcal{V}_{\Gamma}\big([I-H_\kappa(t_i)]b_i)(x)\geq \lambda/4\Big\}\Big).
\end{eqnarray}

By the $L^2$ boundedness of $\mathcal{V}_\Gamma$ and Chebyshev's inequality again,
\begin{eqnarray*}
\mu_\kappa\Big(\Big\{x\in\R^d: \mathcal{V}_{\Gamma}\Big(\sum_i H_\kappa(t_i)b_i\Big)(x)\geq \lambda/4\Big\}\Big)
&\leq&\frac{c}{\lambda^2}\Big\|\mathcal{V}_{\Gamma}\Big(\sum_i H_\kappa(t_i)b_i\Big)\Big\|^2_{L^2(\mu_\kappa)}\\
&\leq&\frac{c}{\lambda^2}\Big\|\sum_i H_\kappa(t_i)|b_i|\Big\|^2_{L^2(\mu_\kappa)},
\end{eqnarray*}
where
\begin{eqnarray*}
\Big\|\sum_i H_\kappa(t_i)|b_i|\Big\|_{L^2(\mu_\kappa)}
&=&\sup_{\|u\|_{L^2(\mu_\kappa)}=1}\Big|\int_{\R^d}u\sum_i H_\kappa(t_i)|b_i|\,\d\mu_\kappa\Big|\\
&=&\sup_{\|u\|_{L^2(\mu_\kappa)}=1}\Big|\sum_i\int_{\R^d}|b_i| H_\kappa(t_i)u\,\d\mu_\kappa\Big|\\
&\leq&\sup_{\|u\|_{L^2(\mu_\kappa)}=1}\sum_i\|b_i\|_{L^1(\mu_\kappa)}\big(\sup_{B_i} H_\kappa(t_i)|u|\big).
\end{eqnarray*}
We \textbf{claim} that, for every $t>0$, $x\in\R^d$ and every nonnegative measurable function $v$ defined on $\R^d$,
\begin{eqnarray}\label{sup-max}
\sup_{y\in B(x,\sqrt{t})}\big(H_\kappa(t)v\big)(y)\leq c\sum_{g\in G}\inf_{z\in B(x,\sqrt{t})}\mathcal{M}(v)(gz),
\end{eqnarray}
where $\mathcal{M}$ is the Hardy--Littlewood maximum operator defined as
$$\mathcal{M}(v)(y)=\sup_{r>0}\frac{1}{\mu_\kappa\big(B(y,r)\big)}\int_{B(y,r)}|v(z)|\,\d\mu_\kappa(z),\quad y\in\R^d.$$
By (b), (c), \eqref{sup-max} and the $G$-invariance of $\mu_\kappa$, we have, for every $u\in L^2(\mu_\kappa)$ with $\|u\|_{L^2(\mu_\kappa)}=1$,
\begin{eqnarray*}
\sum_i\|b_i\|_{L^1(\mu_\kappa)}\big(\sup_{B_i} H_\kappa(t_i)|u|\big)
&\leq&c\lambda\sum_i\mu_{\kappa}(B_i)\sum_{g\in G}\inf_{x\in B_i}\mathcal{M}(u)(gx)\\
&\leq&c\lambda\sum_i\sum_{g\in G}\int_{B_i}\mathcal{M}(u)(gx)\,\d\mu_\kappa(x)\\
&\leq&c\lambda\sum_{g\in G}\sqrt{\mu_\kappa\big(\cup_i B_i\big)}~\|\mathcal{M}(u)\|_{L^2(\mu_\kappa)}\\
&\leq&c\sqrt{\lambda\|f\|_{L^1(\mu_\kappa)}},
\end{eqnarray*}
since $\mathcal{M}$ is bounded in $L^2(\mu_\kappa)$. Hence
\begin{eqnarray}\label{proof-main-4}
\mu_\kappa\Big(\Big\{x\in\R^d: \mathcal{V}_{\Gamma}\Big(\sum_i H_\kappa(t_i)b_i\Big)(x)\geq \lambda/4\Big\}\Big)
\leq \frac{c}{\lambda}\|f\|_{L^1(\mu_\kappa)}.
\end{eqnarray}

Now we start to prove the claim, i.e., \eqref{sup-max}. Let $y\in B(x,\sqrt{t})$. By \eqref{kernel-bd}, we have
\begin{eqnarray*}
\big(H_\kappa(t)v\big)(y)&\leq& C\int_{\R^d}\frac{e^{-c\rho(y,z)^2/t}}{\mu_\kappa\big(B(y,\sqrt{t})\big)}v(z)\,\d\mu_\kappa(z)\\
&\leq& C\sum_{g\in G}\int_{\R^d}\frac{e^{-c|gy-z|^2/t}}{\mu_\kappa\big(B(gy,\sqrt{t})\big)}v(z)\,\d\mu_\kappa(z).
\end{eqnarray*}
For any fixed $g\in G$, let $E_1=B(gx,4\sqrt{t})$ and $E_j=B(gx,2^{j+1}\sqrt{t})\setminus  B(gx,2^j\sqrt{t})$, for $j=2,3,\cdots$. Since $y\in B (x,\sqrt{t})$, we see that for any $z\in E_j$, $|y-x|<\sqrt{t}$, $2^j\sqrt{t}\leq|gx-z|<2^{j+1}\sqrt{t}$, $j=1,2,\cdots$. Then the triangular inequality implies that $|gy-z|\geq|z-gx|-|g(y-x)|=|z-gx|-|y-x|\geq 2^{j-1}\sqrt{t}$, $j=1,2,\cdots$. Thus, for every $y\in B(x,\sqrt{t})$, since $B(gx,2^{j+1}\sqrt{t})\subset B(gy,2^{j+1}\sqrt{t}+|g(x-y)|)\subset B(gy,2^{j+2}\sqrt{t})$, we have
\begin{eqnarray*}
\big(H_\kappa(t)v\big)(y)&\leq& C\sum_{g\in G}\sum_{j=1}^\infty
\int_{E_j}\frac{e^{-c4^{j-1}}}{\mu_\kappa\big(B(gy,\sqrt{t})\big)}v(z)\,\d\mu_\kappa(z)\\
&\leq&C\sum_{g\in G}\sum_{j=1}^\infty e^{-c4^{j-1}}\frac{\mu_\kappa\big(B(gy,2^{j+2}\sqrt{t})\big)}{\mu_\kappa\big(B(gy,\sqrt{t})\big)}\times\\
&&\quad\frac{1}{\mu_\kappa\big(B(gx,2^{j+1}\sqrt{t})\big)}\int_{B(gx,2^{j+1}\sqrt{t})}v(z)\,\d\mu_\kappa(z)\\
&\leq& C\sum_{g\in G}\sum_{j=1}^\infty e^{-c4^{j-1}}  2^{(j+2)d_\kappa}\inf_{z\in B(x,\sqrt{t})}\mathcal{M}(v)(g z)\\
&\leq& C\sum_{g\in G}\inf_{z\in B(x,\sqrt{t})}\mathcal{M}(v)(g z),
\end{eqnarray*}
where the right inequality of \eqref{doubling} is used. We complete the proof of the claim.

\textbf{(3)} It remains to estimate the last term of \eqref{proof-main-3}. For notational simplicity, for each $l$, we let $2B_l^\rho=B^\rho(x_l,2\sqrt{t_l})$ and $(2B_l^\rho)^c=\R^d\setminus 2B_l^\rho$ in the following proof. Then
\begin{eqnarray*}
&&\mu_\kappa\Big(\Big\{x\in\R^d: \sum_i\mathcal{V}_{\Gamma}\big([I-H_\kappa(t_i)]b_i)(x)\geq \lambda/4\Big\}\Big)\\
&\leq&\sum_l\mu_\kappa(2B_l^\rho)+\mu_\kappa\Big(\Big\{x\in\cap_l (2B_l^\rho)^c: \sum_i\mathcal{V}_{\Gamma}\big([I-H_\kappa(t_i)]b_i)(x)\geq \lambda/4\Big\}\Big)\\
&=:&\sum_l\mu_\kappa(2B_l^\rho)+J.
\end{eqnarray*}
Note that $2B_l^\rho=\cup_{g\in G}gB(x_l,2\sqrt{t_l})$. Since $\mu_\kappa$ is $G$-invariant, by (c) and the right inequality in \eqref{doubling}, we derive that
\begin{eqnarray*}
&&\sum_l\mu_\kappa(2B_l^\rho)\leq \sum_l|G|\mu_\kappa\big(B(x_l,2\sqrt{t_l})\big)\\
&\leq&\sum_l|G| 2^{d_\kappa}\mu\big(B(x_l,r_l)\big)\leq \frac{c}{\lambda}\|f\|_{L^1(\mu_\kappa)}.
\end{eqnarray*}
Since $b_i$ is supported in $B_i$ for each $i$ by (b), it is easy to see that
\begin{eqnarray*}
J&\leq&\frac{4}{\lambda}\sum_i\int_{\cap_l(2B^\rho_l)^c}\mathcal{V}_\Gamma\big([I-H_\kappa(t_i)]b_i\big)\,\d\mu_\kappa\\
&=&\frac{4}{\lambda}\sum_i\int_{\cap_l(2B^\rho_l)^c}\!\!\bigg(\int_0^\infty \Gamma\Big(\int_{B_i}[h_s(\cdot,y)-h_{s+t_i}(\cdot,y)]b_i(y)\,\d\mu_\kappa(y)\Big)(x)\,\d s\bigg)^{1/2}\,\d\mu_\kappa(x)\\
&\leq&\frac{4\sqrt{2}}{\lambda}\sum_i \int_{B_i} \int_{(2B^\rho_i)^c}\!\! \Big(\int_0^\infty\Gamma\big( h_s(\cdot,y)- h_{s+t_i}(\cdot,y)\big)(x)\,\d s\Big)^{1/2}\,\d\mu_\kappa(x)   |b_i(y)|\,\d\mu_\kappa(y),
\end{eqnarray*}
where the last inequality can be check directly by the explicit express of $\Gamma$ (see \eqref{Gamma-op}). For each $i$ and every $y\in\R^d$, let
$$J_i(y)= \int_{(2B_i^\rho)^c}\Big(\int_0^\infty \Gamma\big( h_s(\cdot,y)- h_{s+t_i}(\cdot,y)\big)(x) \,\d s\Big)^{1/2}\,\d\mu_\kappa(x).$$
Then
$$J\leq\frac{c}{\lambda}\sum_i \int_{B_i}J_i(y) |b_i(y)|\,\d\mu_\kappa(y). $$
So, by (b) and (c), it suffices to prove that, there exists a positive constant $c$ such that, for each $i$,
\begin{eqnarray*}\label{proof-main-J}\sup_{y\in B_i}J_i(y)\leq c.\end{eqnarray*}

For $m=0,1,2,\cdots$ and $y\in\R^d$, let
\begin{eqnarray*}
J_i^m(y)=\int_{(2B_i^\rho)^c}\Big(\int_{mt_i}^{(m+1)t_i} \Gamma\big( h_s(\cdot,y)- h_{s+t_i}(\cdot,y)\big)(x) \,\d s\Big)^{1/2}\,\d\mu_\kappa(x).
\end{eqnarray*}

\textbf{(i)} Firstly, for $m=1,2,\cdots$ and $y\in B_i$, we estimate $J_i^m(y)$. By the Cauchy--Schwarz inequality, we get
\begin{eqnarray}\label{proof-main-3-1}
J_i^m(y)&=&\int_{(2B_i^\rho)^c}\!\!\bigg(\int_{mt_i}^{(m+1)t_i}\Gamma\big( h_s(\cdot,y)- h_{s+t_i}(\cdot,y)\big)(x)\exp\Big(2\delta\frac{\rho(x,x_i)^2}{mt_i}\Big)\,\d s\bigg)^{1/2}\cr
&&\times\exp\Big(-\delta\frac{\rho(x,x_i)^2}{mt_i}\Big)\,\d\mu_\kappa(x)\cr
&\leq&\sqrt{\tilde{J}_i^m(y)}~\Big\{\int_{(2B_i^\rho)^c} \exp\Big(-2\delta\frac{\rho(x,x_i)^2}{mt_i}\Big)\,\d\mu_\kappa(x)\Big\}^{1/2},
\end{eqnarray}
where we have set
$$\tilde{J}_i^m(y)=\int_{\R^d}\int_{mt_i}^{(m+1)t_i}\Gamma\big( h_s(\cdot,y)- h_{s+t_i}(\cdot,y)\big)(x)\exp\Big(2\delta\frac{\rho(x,x_i)^2}{mt_i}\Big)\,\d s\,\d\mu_\kappa(x).$$
Lemma \ref{lemma-1} implies that
\begin{eqnarray}\label{proof-main-3-2}
\int_{(2B_i^\rho)^c} \exp\Big(-2\delta\frac{\rho(x,x_i)^2}{mt_i}\Big)\,\d\mu_\kappa(x)\leq c\mu\big(B(x_i,\sqrt{mt_i})\big)e^{-\delta/(m\sqrt{t_i})}.
\end{eqnarray}
Since
$$\partial_u h_u(x,y)=\Delta_\kappa h_u(\cdot,y)(x),$$
we have
\begin{eqnarray*}
\tilde{J}_i^m(y)&=&\int_{\R^d}\!\int_{mt_i}^{(m+1)t_i}\Gamma\bigg(\int_s^{s+t_i}  \Delta_\kappa h_u(\cdot,y)\,\d u\bigg)(x)\exp\Big(2\delta\frac{\rho(x,x_i)^2}{mt_i}\Big)\,\d s\,\d\mu_\kappa(x)\\
&\leq&\int_{\R^d}\!\int_{mt_i}^{(m+1)t_i}\Big(t_i\int_s^{s+t_i}\Gamma\big(\Delta_\kappa h_u(\cdot,y)\big)(x) \,\d u\Big)\exp\Big(2\delta\frac{\rho(x,x_i)^2}{mt_i}\Big)\,\d s\,\d\mu_\kappa(x)\\
&=&t_i\int_{mt_i}^{(m+1)t_i}\!\!\int_s^{s+t_i}\bigg( \int_{\R^d}\Gamma\big(\Delta_\kappa h_u(\cdot,y))(x)\exp\Big(2\delta\frac{\rho(x,x_i)^2}{mt_i}\Big)\,\d\mu_\kappa(x)\bigg)\,\d u\,\d s,
\end{eqnarray*}
where we applied the Cauchy--Schwarz inequality in the second inequality and Fubini's theorem in the last equality. Since $s\leq u\leq s+t_i$, $mt_i\leq s\leq (m+1)t_i$, we get $t_i^{-1} \leq(m+2)u^{-1}$. Since $y\in B_i$, and for every $g\in G$, $|gx-x_i|\leq|gx-y|+|y-x_i|$, we have
$$\rho(x,x_i)\leq\rho(x,y)+|y-x_i|<\rho(x,y)+\sqrt{t_i}.$$
Hence
\begin{eqnarray*}
\tilde{J}_i^m(y)&\leq&t_i\int_{mt_i}^{(m+1)t_i}\!\!\!\int_s^{s+t_i}\!\!\Big[ \int_{\R^d}\Gamma\big(\Delta_\kappa h_u(\cdot,y))(x)\exp\Big(2\delta\frac{2\rho(x,y)^2+2t_i}{mt_i}\Big)\,\d\mu_\kappa(x)\Big]\,\d u\,\d s\\
&\leq& c t_i\int_{mt_i}^{(m+1)t_i}\!\!\!\int_s^{s+t_i}\!\!\Big[ \int_{\R^d}\Gamma\big(\Delta_\kappa h_u(\cdot,y))(x)\exp\Big(12\delta\frac{\rho(x,y)^2}{u}\Big)\,\d\mu_\kappa(x)\Big]\,\d u\,\d s.
\end{eqnarray*}
Applying Lemma \ref{lemma-2}, we deduce that, for small enough $\delta>0$,
\begin{eqnarray}\label{proof-main-3-3}
\tilde{J}_i^m(y)&\leq& c t_i\int_{mt_i}^{(m+1)t_i}\!\!\!\int_s^{s+t_i} \frac{1}{u^3\mu\big(B(y,\sqrt{u})\big)}\,\d u\,\d s\cr
&\leq&c t_i^2\int_{mt_i}^{(m+1)t_i} \frac{1}{s^3\mu\big(B(y,\sqrt{s})\big)}\,\d s\cr
&\leq&\frac{c}{m^3\mu\big(B(y,\sqrt{mt_i})\big)}.
\end{eqnarray}
Thus, combining \eqref{proof-main-3-1}, \eqref{proof-main-3-2} and \eqref{proof-main-3-3} with $B(x_i,\sqrt{mt_i})\subset B(y,\sqrt{mt_i}+|x_i-y|)\subset B(y,(\sqrt{m}+1)t_i)$, $y\in B_i$, we have,  by \eqref{doubling},
\begin{eqnarray}\label{proof-main-3-4}
J_i^m(y)&\leq& c\bigg(\frac{\mu\big(B(x_i,\sqrt{mt_i})\big)e^{-\delta/(m\sqrt{t_i})}}{m^3\mu\big(B(y,\sqrt{mt_i})\big)}\bigg)^{1/2}\cr
&\leq&\frac{c}{m^{3/2}}\Big(\frac{1+\sqrt{m}}{\sqrt{m}}\Big)^{d_\kappa/2}\leq\frac{c}{m^{3/2}},
\end{eqnarray}
for every $y\in B_i$ and $m=1,2,\cdots.$

\textbf{(ii)} Secondly, for $y\in B_i$, we estimate $J_i^0(y)$. Similar as the the approach to estimate $J_i^m(y)$ in \textbf{(i)}, we have
\begin{eqnarray*}
J_i^0(y)\leq\sqrt{\tilde{J}_i^0(y)}~\Big\{\int_{(2B_i^\rho)^c} \exp\Big(-2\delta\frac{\rho(x,x_i)^2}{t_i}\Big)\,\d\mu_\kappa(x)\Big\}^{1/2},
\end{eqnarray*}
where  we have let
\begin{eqnarray*}
\tilde{J}_i^0(y)=\int_{(2B_i^\rho)^c}\int_{0}^{t_i}\Gamma\big( h_s(\cdot,y)(x)- h_{s+t_i}(\cdot,y)\big)(x)\exp\Big(2\delta\frac{\rho(x,x_i)^2}{t_i}\Big)\,\d s\,\d\mu_\kappa(x).
\end{eqnarray*}
Again, by Lemma \ref{lemma-1},
\begin{eqnarray*}\label{proof-main-3-5}
\int_{(2B_i^\rho)^c} \exp\Big(-2\delta\frac{\rho(x,x_i)^2}{t_i}\Big)\,\d\mu_\kappa(x)\leq c\mu(B_i).
\end{eqnarray*}
Note that $y\in B_i$. Since for every $g\in G$, $|gx-x_i|\leq|gx-y|+|y-x_i|$, we have
$$\rho(x,x_i)\leq\rho(x,y)+|y-x_i|<\rho(x,y)+\sqrt{t_i}.$$
Then, for every $x\in(2B_i^\rho)^c$, $2\sqrt{t_i}\leq\rho(x,x_i)\leq\rho(x,y)+\sqrt{t_i}$; hence, $\rho(x,y)\geq\sqrt{t_i}$, which implies that $(2B_i^\rho)^c\subset\R^d\setminus B^\rho(y,\sqrt{t_i}).$ Hence
\begin{eqnarray*}
\tilde{J}_i^0(y)&\leq& t_i\int_0^{t_i}\!\!\!\int_{s}^{s+t_i}\!\!\!\int_{(2B_i^\rho)^c}\Gamma\big(\Delta_\kappa h_u(\cdot,y)\big)(x)\exp\Big(2\delta\frac{\rho(x,x_i)^2}{t_i}\Big)\,\d\mu_\kappa(x)\,\d u\,\d s\\
&\leq& t_i\int_0^{t_i}\!\!\!\int_{s}^{s+t_i}\!\!\!\int_{(2B_i^\rho)^c}\Gamma\big(\Delta_\kappa h_u(\cdot,y)\big)(x)\exp\Big(2\delta\frac{2\rho(x,y)^2+2t_i}{t_i}\Big)\,\d\mu_\kappa(x)\,\d u\,\d s\\
&\leq&ct_i\int_0^{t_i}\!\!\!\int_{s}^{s+t_i}\!\!\!\int_{\R^d\setminus B^\rho(y,\sqrt{t_i})}\Gamma\big(\Delta_\kappa h_u(\cdot,y)\big)(x)\exp\Big(8\delta\frac{\rho(x,y)^2}{u}\Big)\,\d\mu_\kappa(x)\,\d u\,\d s.
\end{eqnarray*}
By Lemma \ref{lemma-2} again, for small enough $\delta>0$, we have
$$\int_{\R^d\setminus B^\rho(y,\sqrt{t_i})}\Gamma\big(\Delta_\kappa h_u(\cdot,y)\big)(x)\exp\Big(8\delta\frac{\rho(x,y)^2}{u}\Big)\,\d\mu_\kappa(x)
\leq\frac{ce^{-ct_i/u}}{u^3\mu_\kappa\big(B(y,\sqrt{u})\big)}.$$
Hence
\begin{eqnarray*}\label{proof-main-3-6}
\tilde{J}_i^0(y)&\leq&ct_i\int_0^{t_i}\!\!\!\int_{s}^{s+t_i} \frac{e^{-ct_i/u}}{u^3\mu_\kappa(B(y,\sqrt{u}))}   \,\d u\,\d s\\
&=&\frac{c}{t_i^2\mu_\kappa\big(B(y,\sqrt{t_i})\big)}\int_0^{t_i}\!\!\!\int_{s}^{s+t_i} \Big(\frac{t_i}{u}\Big)^3\frac{\mu_\kappa\big(B(y,\sqrt{t_i})\big)}{\mu_\kappa\big(B(y,\sqrt{u})\big)}e^{-ct_i/u} \,\d u\,\d s\\
&\leq&\frac{c}{t_i^2\mu_\kappa\big(B(y,\sqrt{t_i})\big)}\int_0^{t_i}\!\!\!\int_{s}^{s+t_i}
\Big(\frac{t_i}{u}\Big)^{3+d_\kappa/2}e^{-ct_i/u}\,\d u\,\d s\\
&\leq&\frac{c}{\mu_\kappa\big(B(y,\sqrt{t_i})\big)},
\end{eqnarray*}
where the last inequality is due to the fact that $\R_+\ni t\mapsto t^{3+d_\kappa/2}e^{-ct}$ is bounded. Thus, by \eqref{doubling}, since $B_i\subset B(y, \sqrt{t_i}+|y-x_i|)\subset B(y,2\sqrt{t_i})$, $y\in B_i$, we have
\begin{eqnarray}\label{proof-main-3-7}
J_i^0(y)\leq c \bigg(\frac{\mu_\kappa(B_i)}{\mu_\kappa\big(B(y,\sqrt{t_i})\big)}\bigg)^{1/2}\leq c2^{d_\kappa/2},\quad y\in B_i.
\end{eqnarray}

Thus, from \eqref{proof-main-3-4} and \eqref{proof-main-3-7}, we obtain that for each $i$,
$$\sup_{y\in B_i}J_i(y)\leq\sum_{m=0}^\infty \sup_{y\in B_i}J_i^m(y)\leq c\Big(1+\sum_{m=1}^\infty \frac{1}{m^{3/2}}\Big)\leq c,$$
which implies that
\begin{eqnarray}\label{proof-main-5}
\mu_\kappa\Big(\Big\{x\in\R^d: \sum_i\mathcal{V}_{\Gamma}\big([I-H_\kappa(t_i)]b_i)(x)\geq \lambda/4\Big\}\Big)\leq\frac{ c}{\lambda}\|f\|_{L^1(\mu_\kappa)}.
\end{eqnarray}

\textbf{(4)} Therefore, combining \eqref{proof-main-1}, \eqref{proof-main-2}, \eqref{proof-main-3}, \eqref{proof-main-4} and \eqref{proof-main-5}, we obtain \eqref{proof-main-0}. The proof of Theorem \ref{main-thm-v} is completed.
\end{proof}

Now Theorem \ref{main-thm-h} can be proved by applying the same method used in the proof of Theorem \ref{main-thm-v}. The main difference lies in part \textbf{(3)} in the above proof, where Lemma \ref{integral-time-der-bd} should be employed instead of Lemma \ref{lemma-2}. We omit details here to save some space.

\subsection*{Acknowledgment}\hskip\parindent
The author would like to thank his colleague Mingfeng Zhao for helpful discussions, and to acknowledge the financial support from the National Natural Science Foundation of China (Grant No. 11831014).


\begin{thebibliography}{a23}
\bibitem{AmSi2012}
B. Amri, M. Sifi: Singular integral operators in Dunkl setting. J. Lie Theory 22 (2012), 723--739.

\bibitem{AmHa2019}
B. Amri, A. Hammi: Dunkl--Schr\"{o}dinger operators. Complex Anal. Oper. Theory 13 (2019), 1033--1058.

\bibitem{AmHa2020}
B. Amri, A. Hammi: Semigroup and Riesz transform for the Dunkl--Schr\"{o}dinger operators. Semigroup Forum 101 (2020), 507--533.

\bibitem{Anker2017}
J.-P. Anker: An introduction to Dunkl theory and its analytic aspects. In \emph{Analytic, algebraic and geometric aspects of differential equations}. Trends in Math., Birkh\"{a}user, Cham, 2017, pp. 3--58.

\bibitem{ADH2019}
J.-P. Anker, J. Dziuba\'{n}ski, A. Hejna: Harmonic Functions, Conjugate Harmonic Functions and the Hardy Space $H^1$ in the Rational Dunkl Setting. J. Fourier Anal. Appl. 25 (2019), 2356--2418.

\bibitem{BE1985} D. Bakry, M. Emery: Diffusions hypercontractives. In \emph{S\'{e}minaire de probabilit\'{e}s XIX, 1983/84}. Lecture Notes in Math. vol. 1123, Springer, Berlin, 1985, pp.177--206.

\bibitem{CW1971}
R. Coifman, G. Weiss: \emph{Analyse Harmornique Non-commutative sur Certains Espaces Homogenes}. Lecture Notes in Math. vol. 242, Springer-Verlag,
Berlin-New York, 1971.

\bibitem{CoulhonDuong1999}
T. Coulhon, X.T. Duong: Riesz transforms for $1\leq p\leq2$. Trans. Amer. Math. Soc. 351 (1999), 1151--1169.

\bibitem{CDD}
T. Coulhon, X.T. Duong, X.D. Li: Littlewood--Paley--Stein functions on complete Riemannian manifolds for $1\le p\le 2$. Studia Math. 154 (2003), 37--57.

\bibitem{Davies1997}
E.B. Davies: Non-Gaussian aspects of heat kernel behaviour. J. London Math. Soc. 55(1) (1997), 105--125.


\bibitem{DaiWang2010}
 F. Dai, H. Wang: A transference theorem for the Dunkl transform and its applications. J. Funct. Anal. 258 (2010), 4052--4074.

\bibitem{DX2015}
F. Dai, Y. Xu: \emph{Analysis on $h$-harmonics and Dunkl transforms}. Advanced Courses in Mathematics CRM Barcelona, edited by Sergey Tikhonov,  Birkh\"{a}user/Springer, Basel, 2015.

\bibitem{Dunkl1988}
C.F. Dunkl: Reflection groups and orthogonal polynomials on the sphere. Math. Z. 197 (1988), 33--60.

\bibitem{Dunkl1989}
C.F. Dunkl: Differential-difference operators associated to reflection groups. Trans. Amer. Math. Soc. 311 (1989), 167--183.

\bibitem{Dunkl1992}
C.F. Dunkl: Hankel transforms associated to finite reflection groups. In \emph{Hypergeometric functions on domains of positivity, Jack polynomials, and applications}. Tampa, FL, 1991. Contemp. Math. 138,  Amer. Math. Soc., Providence, RI, 1992, pp. 123--138.

\bibitem{DunklXu2014}
C.F. Dunkl, Y. Xu: \emph{Orthogonal polynomials of several variables}. Encyclopedia of Mathematics and its Applications 155, Cambridge University Press, Cambridge, Second edition, 2014.

\bibitem{DH2020}
J. Dziuba\'{n}ski, A. Hejna: Upper and lower bounds for Littlewood-Paley square functions in the Dunkl setting. Preprint (2020), arXiv:2005.00793v2.

\bibitem{DziPre2018}
J. Dziuba\'{n}ski, M. Preisner: Hardy spaces for semigroups with Gaussian bounds. Annali di Matematica 197(3) (2018), 965--987.

\bibitem{JLZ2016}
R. Jiang, H. Li, H. Zhang: Heat kernel bounds on metric measure spaces and some applications. Potential Anal. 44 (2016), 601--627.

\bibitem{LZ2020}
H. Li, M. Zhao: Dimension-free square function estimates for Dunkl operators. Preprint (2020), arXiv:2003.11843. To appear in Math. Nachr., doi: 10.1002/mana.202000210.

\bibitem{LZL2017}
Jianquan Liao, Xiaoliang Zhang, Zhongkai Li: On Littlewood--Paley functions associated with the Dunkl operator. Bull. Aust. Math. Soc. 96 (2017), 126--138.

\bibitem{ML2019}
M. Maslouhi, E. Lamine: On the Generalized Ornstein-Uhlenbeck Semigroup. Preprint (2019), arXiv:1907.13474v3.

\bibitem{Rosler1998}
M. R\"{o}sler: Generalized Hermite polynomials and the heat equation for Dunkl operators. Comm. Math. Phys. 192 no. 3 (1998), 519--542.

\bibitem{Rosler2003a}
M. R\"{o}sler: A positive radial product formula for the Dunkl kernel. Trans. Amer. Math. Soc. 355 (2003), 2413--2438.

\bibitem{Rosler2003}
M. R\"{o}sler: Dunkl operators: Theory and Applications. In \emph{Orthogonal Polynomials and Special Functions}, Leuven 2002, edited by E. Koelink, W. Van Assche. Lecture Notes in Math. vol. 1817, Springer, Berlin, 2003, pp. 93--135.

\bibitem{RosVoi1998}
M. R\"{o}sler, M. Voit: Markov processes related with Dunkl operators. Adv. App. Math. 21 (1998), 575--643.

\bibitem{Soltani2005}
F. Soltani: Littlewood--Paley operators associated with the Dunkl operator on $\R$. J. Funct. Anal. 221 (2005), 205--225.

\bibitem{Soltani05}
F. Soltani: Littlewood-Paley g-function in the Dunkl analysis on $\R^d$. J. Ineq. Pure and Appl. Math. 6 (2005), no. 3.

\bibitem{St1970}
E.M. Stein: \emph{Singular Integrals and Differentiability Properties of Functions}. Princeton Mathematical Series, No. 30, Princeton
Univ. Press, Princeton, 1970.

\bibitem{Stein70}
E.M. Stein: \emph{Topics in Harmonic Analysis Related to the Littlewood--Paley Theory}. Ann. of Math. Stud. 63, Princeton
    Univ. Press, Princeton, 1970.

\bibitem{Stein1982}
E.M. Stein: The development of square functions in the work of A. Zygmund. Bull. Amer. Math. Soc. 7 (1982), 359--376.
\end{thebibliography}
\end{document}